\documentclass[12pt]{amsart}
\usepackage{amsmath,amssymb,amsfonts,enumerate,amsthm, amscd,}
\usepackage{txfonts}
\newtheorem{theorem}{Theorem}[section]
\newtheorem{proposition}[theorem]{Proposition}
\newtheorem{lemma}[theorem]{Lemma}

\newtheorem{corollary}[theorem]{Corollary}

\theoremstyle{definition}

\theoremstyle{remark}
\newtheorem{remark}[theorem]{Remark}

\numberwithin{equation}{section}

\catcode`\ç=13
\defç{\c{c}}
\catcode`\é=13
\defé{\'e}
\catcode`\à=13
\defà{\`a}
\catcode`\è=13
\defè{\`e}
\catcode`\â=13
\defâ{\^a}
\catcode`\ù=13
\defù{\`u}
\catcode`\ê=13
\defê{\^e}
\catcode`\î=13
\defî{\^\i}
\catcode`\ô=13
\defô{\^o}
\begin{document}

\title{Self injective property in amalgamated algebra along an ideal}

\author[Najib Mahdou]{Najib Mahdou}
\address{Department of Mathematics, Faculty of Science and Technology of Fez, Box 2202, University S. M.
Ben Abdellah Fez, Morocco}
 \email{mahdou@hotmail.com}

\author[Moutu Abdou Salam Moutui]{Moutu Abdou Salam Moutui}
\address{Department of Mathematics, Faculty of Science and Technology of Fez, Box 2202, University S. M.
Ben Abdellah Fez, Morocco}

%\curraddr{Department of Mathematics, Faculty of Science and
%Technology of Fez,\\ Box 2202, University S. M. Ben Abdellah Fez,
%Morocco}
 \email{moutu\_2004@yahoo.fr}


%    General info


\begin{abstract} Let $f: A\rightarrow B$ be a ring homomorphism and let $J$ be an ideal of $B$. In this paper, we investigate the transfer of
self-injective property to the amalgamation of $A$ with $B$
along $J$ with respect to $f$ (denoted by $A\bowtie^fJ),$
introduced and studied by D'Anna, Finocchiaro and Fontana in 2009.
We give also a characterization of $A\bowtie^fJ$ to be
quasi-Frobenius.\smallskip
\end{abstract}

\maketitle

\section{Introduction} All rings considered in this paper are assumed to be commutative,
 and have identity element and all modules are unitary.\\

Let $A$ and $B$ be two rings with unity, let $J$ be an ideal of
$B$ and let $f: A\rightarrow B$ be a ring homomorphism. In this
setting, we can consider the following subring of $A\times B$:
\begin{center} $A\bowtie^{f}J: =\{(a,f(a)+j)\mid a\in A,j\in
J\}$\end{center} called \emph{the amalgamation of $A$ and $B$
along $J$ with respect to $f$} (introduced and studied by D'Anna,
 Finacchiaro, and  Fontana in \cite{AFF1, AFF2}). This
construction is a generalization of \emph{the amalgamated
duplication of a ring along an ideal} (introduced and studied by
D'Anna and Fontana in \cite{A, AF1, AF2}).  Moreover, other
classical constructions (such as the $A+XB[X]$, $A+XB[[X]]$, and
the $D+M$ constructions) can be studied as particular cases of the
amalgamation (\cite[Examples 2.5 and 2.6]{AFF1}) and other
classical constructions, such as the Nagata's idealization (cf.
\cite[page 2]{Nagata}), and the CPI extensions  are strictly
related to it (\cite[Example 2.7 and Remark 2.8]{AFF1}). On the
other hand, the amalgamation is related to a construction proposed
by Anderson in \cite{a1} and motivated by a classical construction
due to Dorroh \cite{do}, concerning the embedding of a ring
without identity in a ring with identity. In \cite{AFF1}, the
authors studied the basic properties of this construction (e.g.,
characterizations for $A\bowtie^{f}J$ to be a Noetherian ring, an
integral domain, a reduced ring) and they characterized those
distinguished pullbacks that can be expressed as an amalgamation.
Moreover, in \cite{AFF2}, they pursued the investigation on the
structure of the rings of the form $A\bowtie^{f}J$, with
particular attention to the prime spectrum, to the chain
properties and to the Krull dimension. \\

Self-injective rings (i.e., rings that are injective modules over
themselves) play an important role in ring theory since they have
connections with several kinds of rings; e.g., quasi-Frobenius
rings, semiprimary rings, and Kasch rings (see \cite{nich}). In
\cite{chhMahtam}, The authors characterize an amalgamated
duplication of a ring $R$ along an ideal $I$, denoted by $R\bowtie
I$ to be self-injective. \\

In this paper, we investigate the transfer of self-injective and
quasi-Frobenius properties to amalgamation $A\bowtie^{f}J$ and so
we generalize \cite{chhMahtam}. \\

\section{Main results}\label{sec:2}

We first give some results of amalgamated algebra along an ideal.
Recall that the modulation of $A$ over $A\bowtie^{f}J$ is given
via the ring map $g : A\bowtie^{f}J\rightarrow A;$
$(a,f(a)+j)\mapsto a$ for all $a\in A$ $,j\in J$. Precisely,
$(a',f(a')+j).a:=a'a$ for each $a,a'\in A$ and $j\in J.$ \\

\begin{proposition}\label{prop0}
Let $(A,B)$ be a pair of rings, $f : A\rightarrow B$ be an
injective ring homomorphism and $J$ be an ideal of $B$. Assume
that $J\subseteq f(A)$. Then following isomorphism of $A-$modules
hold: $$Hom_{A\bowtie^{f}J}(A,A\bowtie^{f}J)\cong J\oplus
Ann_B(J).$$
\end{proposition}

\begin{proof}
Consider $\phi\in$ $Hom_{A\bowtie^{f}J}(A,A\bowtie^{f}J)$ and set
$\phi(1)=(a,f(a)+x)$ with $a\in A$ and $x\in J$. So, for each
$j\in J,$
$(0,0)=\phi(0)=\phi((0,j).1)=(0,j)\phi(1)=(0,j)(a,f(a)+x)=((0,j(f(a)+x))$.
Hence, $f(a)+x\in Ann_B(J)$. Consequently, by the previous
considerations, we have the following maps: \\ $\psi :
Hom_{A\bowtie^{f}J}(A,A\bowtie^{f}J) \longrightarrow J\oplus
Ann_B(J)$ and \\
$g \longmapsto (x,f(a)+x)$ where $g(1)=(a,f(a)+x).$ \\ One can
easily check that $\psi$ is an injective homomorphism of
$A-$modules since $f$ is injective. It remains to show that $\psi$
is surjective. Let $(x,j)\in J\oplus Ann_B(J).$ Since $J\subseteq
f(A),$ there exist $x',j'\in f^{-1}(J)$ such that $f(j')=j$ and
$f(x')=x$. Consider the $A\bowtie^{f}J-$morphism defined by $g \in
Hom_{A\bowtie^{f}J}(A,A\bowtie^{f}J)$ by setting,
$g(1)=(x'-j',x)$. Explicitly,
$g(a)=g((a,f(a)+j).1)=(a,f(a)+j)(x'-j',x)=(a(x'-j'),f(a)x)$. And
so, $\psi(g)=(x'-j',x).$ Thus, $\psi$ is an isomorphism of
$A-$modules.
\end{proof}

\bigskip

Let $f : A\rightarrow B$ be a ring homomorphism and $J$ be an
ideal of $B$. Consider the canonical (multiplication) $B$-map $\pi
: B\rightarrow Hom_B(J,J)$ (defined by setting $\pi(b)(j)=bj$ for
each $b\in B$ and $j\in J$). It is clear that $ker(\pi)=ann_B(J)$.
\\

\begin{proposition}\label{propo1}
Let $(A,B)$ be a pair of rings, $f: A\rightarrow B$ be an injective ring homomorphism and $J$ be an ideal of $B$ such that $J\subseteq f(A)$ and : \\
$(1)$ the short exact sequence of $B-$modules: \\
$(*)$ $0\rightarrow ann_B(J)\hookrightarrow B\xrightarrow{\pi} Hom_{B}(J,J)\rightarrow 0$ is exact and splits and,\\
$(2)$ $Hom_B(J,ann_B(J))=0$.\\
Then, $Hom_A(A\bowtie^{f}J,J\oplus ann_B(J))$ is isomorphic to $A\bowtie^{f}J$ as $A\bowtie^{f}J-$module.
\end{proposition}

\begin{proof}
Since the short sequence $(*)$ is exact and splits, there exists an $A$-homomorphism $\pi^{-1} : Hom_{B}(J,J)\rightarrow B$ such that $\pi\circ\pi^{-1}$ is the
identity on $Hom_{B}(J,J)$. Consider the $A\bowtie^{f}J-$homomorphism
$$\psi : A\bowtie^{f}J\rightarrow Hom_{A}(A\bowtie^{f}J,J\oplus ann_B(J))$$
defined by  $\psi((1,1)):=\phi_{(1,1)}$ where
$\phi_{(1,1)}((y,f(y)+j))=(j,(1-\pi^{-1}(\pi(1)))f(y))$ for all $y\in A$. It is easy to see that $1-\pi^{-1}(\pi(1))\in Ann_B(J).$
Explicitly, for all $a\in A$ and $i\in J,$ $\psi((a,f(a)+i)):=\phi_{(a,f(a)+i)}$ where
\begin{eqnarray*}
\phi_{(a,f(a)+i)}(y,f(y)+j) & = &\psi((a,f(a)+i))((y,f(y)+j))\\
                            & = & (a,f(a)+i).\psi((1,1))((y,f(y)+j))\\
                            & = & (a,f(a)+i).\phi_{(1,1)}((y,f(y)+j))\\
                            & = & \phi_{(1,1)}((a,f(a)+i)(y,f(y)+j))\\
                            & = & \phi_{(1,1)}((ay,f(a)f(y)+f(a)j+i(f(y)+j))\\
                            & = & (f(a)j+i(f(y)+j),(1-\pi^{-1}(\pi(1)))f(a)f(y))
\end{eqnarray*}
Recall that the natural structure of $A\bowtie^{f}J-$module on
$Hom_A(A\bowtie^{f}J,J\oplus ann_B(J)),$  is defined by the scalar
multiplication by
$(a,f(a)+i)\phi((y,f(y)+j))=\phi((a,f(a)+i)(y,f(y)+j))$. Consider
$(a,f(a)+i)\in A\bowtie^{f}J$ such that $\phi_{(a,f(a)+i)}=0.$ So,
$(0,0)=\phi_{(a,f(a)+i)}(1,1)=(i,f(a)-\pi^{-1}(\pi(f(a))))$.
Therefore, $i=0$ and $f(a)=\pi^{-1}(\pi(f(a)))$. Moreover,
$(0,0)=\phi_{(a,f(a)+i)}((0,j))=(f(a)j+ij,0)=(f(a)j,0)$ for all
$j\in J$. Consequently, $f(a)\in ann_B(J)$. And so $\pi(f(a))=0$
and $f(a)=\pi^{-1}(\pi(f(a)))=0$. Using the fact $f$ is injective,
we obtain $a=0$. Hence, $\psi$ is injective. \\
Now, we prove that
$\psi$ is surjective. Let $\phi\in Hom_{A}(A\bowtie^{f}J,J\oplus
ann_B(J))$. For all $j'$. The set
$\phi((-j',0))=(\sigma_1(f(j')),\sigma_2(f(j')))$. It is clear
that $\sigma_1\in Hom(J,J)$ and $\sigma_2\in Hom_{B}(J,ann_B(J))$.
And so $\sigma_2=0$. Moreover, set $k:=\pi^{-1}(\sigma_1)$. For
all $j'\in f^{-1}(J),$ $\phi((-j',0))=(kj,0)$ with $f(j')=j$.
Also, set $\phi((1,1))=(i,x)$. Finally, set $f(a)=k+x$. Thus,
since $x\in ann_B(J)=ker(\pi),$ we have
$\pi^{-1}(\pi(f(a)))=\pi^{-1}(\pi(k+x))=\pi^{-1}(\pi(k))=\pi^{-1}(\pi(\pi^{-1}(\sigma_1)))=\pi^{-1}(\sigma_1)=k$
and $f(a)j=kj$ for all $j\in J$.
Consequently, for each $a\in A$ and $j\in J,$ using the fact $J\subset f(A),$ there exists $j'\in f^{-1}(J)$ such that $f(j')=j$. We have :\\

$
\phi((y,f(y)+j))= \phi((y,f(y))+(0,j))\\
                = \phi((y,f(y))+\phi((j',j)+(-j',0))\\
                = (y,f(y))\phi((1,1))+(j',j)\phi((1,1))+\phi(-j',0) \\
                =  (y,f(y))(i,x)+ (j',j)(i,x)+(kj,0)\\
                =  (f(y)i,f(y)x)+ (ji,0)+(kj,0)\\
                = (i(f(y)+j)+kj,f(y)x)\\
                =(i(f(y)+j)+kj,f(y)(f(a)-k))\\
                = (i(f(y)+j)+f(a)j,(f(a)-\pi^{-1}(\pi(f(a))))f(y))\\
                = \phi_{(a,f(a)+i)}(y,f(y)+j).$\\

Hence, $\psi$ is an isomorphism of $A\bowtie^{f}J-$modules, as desired. \\
\end{proof}

\bigskip

\begin{remark}\label{r2} 
In particular, the conditions of the previous proposition are satisfied  when $J=eB$ where $e$ is a non zero idempotent element of $B$. Indeed, for each $h\in Hom_B(eB,eB)$ and each $b\in B,$ $h(eb)=h(e^2b)=h(e)eb$. Hence, the canonical (multiplication) $\pi : B\rightarrow Hom_B(eB,eB)$ is surjective. Moreover, $Hom_B(eB,eB)\cong eB$ and so it is a projective module. Thus, the sequence :\\
$(*)$ $0\rightarrow ann_B(eB)\hookrightarrow B\xrightarrow{\pi} Hom_{B}(eB,eB)\rightarrow 0$ is exact and splits. On the other hand, for each $g\in Hom_B(eB,ann_B(eB)),$ $g(eb)=g(e^2b)=g(eb)e=0$.
\end{remark}

\bigskip

The main result  of this paper is the following: \\

\begin{theorem}\label{thm0}
Let $(A,B)$ be a pair of rings, $f: A\rightarrow B$ be an injective ring homomorphism and $J$ be an ideal of $B$ such that $J\subseteq f(A)$. Then $A\bowtie^{f}J$ is a self-injective ring if and only if $B$ is an $A-$module injective and there exists an idempotent element $e\in B$ such that $J=eB$.
\end{theorem}

\begin{proof}
Assume that $A\bowtie^{f}J$ is a self-injective ring. By
Proposition \ref{prop0}, $J$ and $ann_B(J)$ are injective
$A$-modules. Consider the short exact sequence of $B$-modules :
$$(*)  0\rightarrow J\hookrightarrow B\xrightarrow{p}
B/J\rightarrow 0.$$

 We have $Ext^{1}_A(B/J,J)=0$. So, $(*)$
splits. Therefore, $B=J\oplus p^{-1}(B/J)$. Consequently, $J$ is a
principal ideal of $B$. Set $1=e+g$ with $e\in J$ and $g\in
p^{-1}(B/J)$. For each $x\in J,$ $x=xe+xg$ and $x-xe=xg\in J\cap
p^{-1}(B/J)=(0)$. Hence, $J=eB$. Moreover, $ann_B(J)=(1-e)B$.
Thus, $B=J\oplus ann_B(J)$ is injective as an $A$-module. \\
Conversely, assume that $B$ is an injective $A-$module and there
exists an idempotent element $e$ such that $J=eB$. It is clear
that $ann_B(J)=(1-e)B$. Thus, $J\oplus ann_B(J)=B$. By Proposition
\ref{propo1} and Remark \ref{r2}, $A\bowtie^{f}J$ is isomorphic
(as $A\bowtie^{f}J-$module) to $Hom_B(A\bowtie^{f}J,J\oplus
ann_B(J))=Hom_A(A\bowtie^{f}J,B)$. Then, since $B$ is an injective
$A-$module, it follows that $A\bowtie^{f}J$ is an injective as
$A\bowtie^{f}J-$module and this completes the proof of Theorem
\ref{thm0}.
\end{proof}

\bigskip

The following Corollaries are consequences of Theorem \ref{thm0}.
\\

\begin{corollary}
Let $A$ be a ring, $B$ be a local ring, $f: A\rightarrow B$ be an
injective ring homomorphism and let $J$ be a non zero proper ideal
of $B$ such that $J\subseteq f(A)$. Then $A\bowtie^{f}J$ is never
a self-injective ring.
\end{corollary}

\begin{proof} Since $B$ is a local ring, then the only idempotent
elements of $B$ are $\{0,1\}$. Hence, using the fact J is a non
zero proper ideal of $B$ and Theorem \ref{thm0}, we obtain the
desired result.
\end{proof}

\bigskip

The following Corollary is a consequence of Theorem \ref{thm0} and is \cite[Theorem 2.4]{chhMahtam}. \\

\begin{corollary}
Let $A$ be a ring and let $I$ be a ideal of $A$. Then $A\bowtie I$
is a self-injective ring if and only if so is A and there exists
an idempotent element $e\in A$ such that $I=eA$.
\end{corollary}

\begin{proof}
It is easy to see that $A\bowtie I=A\bowtie^f J$ where $f$ is the
identity map of $A$, $B=A,$ $J=I$. One can easily check that
$I\subset f(A)$ and $f$ is injective. So, by Theorem \ref{thm0},
$A\bowtie I$ is a self-injective ring if and only if $B=A$ is an
$A-$module injective and and there exists an idempotent element
$e\in A$ such that $J=I=eA$ and this completes the proof.
\end{proof}

\bigskip

  Now, we give a
characterization of $A\bowtie^{f}J$ to be quasi-Frobenius. Recall
that a ring is quasi-Frobenius if and
only if it is Noetherian and self-injective.\\

\begin{theorem}\label{thm1}
Let $(A,B)$ be a pair of rings, $f: A\rightarrow B$ be an injective ring homomorphism and $J$ be an ideal of $B$ such that $J\subseteq f(A)$.
Then $A\bowtie^{f}J$ is quasi-Frobenius if and only if so is $A$, $f(A)+J$ is Noetherian, $B$ is an $A-$module injective and there exists an
idempotent element $e\in B$ such that $J=eB$.
\end{theorem}

Before proving this Theorem, we need the following Lemmas. \\

\begin{lemma}\label{charquasi}\cite[Theorem 1.50, 7.55 and
7.56]{nich} \\
For a ring $A$, the following statements are equivalent :\\
$(1)$ $A$ is quasi-Frobenius.\\
$(2)$ $A$ is Artinian and self-injective.\\
$(3)$ Every projective $A$-module is injective.\\
$(4)$ Every injective $A$-module is projective.\\
$(5)$ $A$ is Noetherian and $Ann_A(Ann_A(J))=J$ for every ideal $J$ of $A$, where $Ann_A(J)$ denotes the annihilator of $J$ in $A$.
\end{lemma}

\bigskip

\begin{lemma}\label{prod}
Let $(A_i)_{i\in I}$ be a family of commutative rings. Then $A=\prod_{i=1}^{i=n}A_i$ is quasi-Frobenius if and only if so are $A_i$ for all $i\in I$.
\end{lemma}

\begin{proof}
Assume that $A=\prod_{i=1}^{i=n}A_i$ is quasi-Frobenius. Let $i\in
I$, using \cite[Proposition 2.6]{benmah}, $G-gldim(A)=0.$ By
\cite[Theorem 3.1]{benmah2}, it follows that $G-gldim(A_i)=0.$
Hence, $A_i$ is quasi-Frobenius for all $i\in I$. Conversely,
assume that for all $i\in I$, $A_i$ is quasi-Frobenuis. By
\cite[Theorem 3.1]{benmah2},
$G-gldim(A)=G-gldim(\prod_{i=1}^{i=n}A_i)=0.$ Hence, by
\cite[Proposition 2.6]{benmah}, $A=\prod_{i=1}^{i=n}A_i$ is
quasi-Frobenius, as desired.
\end{proof}

\bigskip

\begin{lemma}\label{quas}
Let $(A,B)$ be a pair of rings, $f: A\rightarrow B$ be a ring
homomorphism and let $J$ be an ideal of $B$. If $A\bowtie^{f}J$ is
quasi-Frobenius, then so is $A$.
\end{lemma}

\begin{proof}
Suppose that $A\bowtie^{f}J$ is quasi-Frobenius. It is easy to see
that if $J=0$, then by \cite[Proposition 5.1 (3)]{AFF1}, $A\cong
\frac{A\bowtie^{f}J}{\{0\}\times \{J\}}$. So, $A\cong
A\bowtie^{f}J$ which is quasi-Frobenius. If $J=B,$ then
$A\bowtie^{f}J=A\times B$. So, by Lemma \ref{prod}, $A$ is
quasi-Frobenius. Now, assume that $J$ is a proper ideal of $B$. By
Lemma \ref{charquasi}, $A\bowtie^{f}J$ is Noetherian and
$Ann_{A\bowtie^{f}J}(Ann_{A\bowtie^{f}J}(L))=L$, for every ideal
$L$ of $A\bowtie^{f}J$ where $Ann_{A\bowtie^{f}J}(-)$ is the
annihilator over $A\bowtie^{f}J$. By \cite[Proposition 5.6]{AFF1},
$A$ is Noetherian. \\
Let $K$ be an ideal of $A$ and our aim is to show that
$Ann_{A}(Ann_{A}(K))=K$. Clearly, $K\subseteq
Ann_{A}(Ann_{A}(K))$. Conversely, let
$K\bowtie^{f}J:=\{(k,f(k)+j)/k\in K$ and $j\in J$$\}$ be an ideal
of $A\bowtie^{f}J$. Using the fact $A\bowtie^{f}J$ is
quasi-Frobenius,
$Ann_{A\bowtie^{f}J}(Ann_{A\bowtie^{f}J}(K\bowtie^{f}J))=K\bowtie^{f}J$.
Let $(y,f(y)+h)\in Ann_{A\bowtie^{f}J}(K\bowtie^{f}J)$. Then,
$\forall k\in K,$ $(y,f(y)+h)(k,f(k))=(0,0)$. Therefore, $y\in
Ann_A(K)$ and $h\in Ann_B(J)$. Now, if $x\in Ann_A(Ann_A(K))$,
then $(y,f(y)+h)(x,f(x))=(0,0)$ and $(x,f(x))\in
Ann_{A\bowtie^{f}J}(Ann_{A\bowtie^{f}J}(K\bowtie^{f}J))=K\bowtie^{f}J$.
Hence, it follows that $x\in K$. Thus, by Lemma \ref{charquasi},
$A$ is quasi-Frobenius, as desired.
\end{proof}

\bigskip

{\parindent0pt {\bf Proof of Theorem \ref{thm1}.\ }} Assume that
$A\bowtie^{f}J$ is quasi-Frobenius. By Lemma \ref{quas}, $A$ is
quasi-Frobenius. Using Lemma \ref{charquasi}, $A\bowtie^{f}J$ is
Notherian. So, by \cite[Proposition 5.6]{AFF1}, $f(A)+J$ is
Noetherian. Since $f$ is injective and $J\subset f(A),$ then by
Theorem \ref{thm0}, $B$ is an $A-$module injective and there
exists an idempotent element $e\in B$ such that $J=eB$, as
desired. \\
Conversely, assume that $A$ is quasi-Frobenius,
$f(A)+J$ is Noetherian, $B$ is an $A-$module injective and there
exists an idempotent element $e\in B$ such that $J=eB$. By
\cite[Proposition 5.6]{AFF1} and Theorem \ref{thm0}, it follows
that $A\bowtie^{f}J$ is quasi-Frobenius and this completes the proof of Theorem \ref{thm1}. \qed\\

\bigskip

The following Corollaries follows immediately from Theorem
\ref{thm1}. \\

\begin{corollary}
Let $A$ be a ring, $B$ be a local ring, $f: A\rightarrow B$ be an injective ring homomorphism and $J$ be a proper ideal of $B$ such that $J\subset f(A)$. Then $A\bowtie^{f}J$ is never quasi-Frobenius.
\end{corollary}

\bigskip

The following Corollary is a consequence of Theorem \ref{thm0} and is \cite[Proposition 2.6]{chhMahtam}. \\

\begin{corollary}
Let $A$ be a ring and $I$ be a ideal of $A$. Then $A\bowtie I$ is
quasi-Frobenius if and only if so is $A$ and there exists an
idempotent element $e\in A$ such that $I=eA$.
\end{corollary}

\bigskip

We end this paper with a characterization for $A\bowtie^{f}J$ to
be quasi-Frobenius in a local setting. For this, we need the
following lemma of
independent interest. \\

\begin{lemma}\label{prop0}
Let $(A,B)$ be a pair of rings, $f : A\rightarrow B$ be a
surjective ring homomorphism and $J$ be an ideal of $B$. Assume
that $Ann_B(J)=0$. Then $ J \cong
Hom_{A\bowtie^{f}J}(A,A\bowtie^{f}J)$.
\end{lemma}

\begin{proof}
By \cite[Proposition 5.1 (3)]{AFF1}, $A\cong \frac{A\bowtie^{f}J}{\{0\}\times \{J\}}$. So, $A$ is a cyclic $A\bowtie^{f}J-$module generated (modulo $(0, J)$). Moreover, for all $a,b\in A$ and $i\in J$, $(a,f(a)+i)b=\pi_{1}((a,f(a)+i))b=ab$ where $\pi_{1}(A\bowtie^{f}J)=A$. Now, for all $j\in J,$
consider the following map defined by :\\
$\phi : J\rightarrow Hom_{A\bowtie^{f}J}(A,A\bowtie^{f}J)$\\
         $j \rightarrow \psi_j$.\\
 where $\psi_{j}: A\rightarrow A\bowtie^{f}J $ defined by $\psi_{j}(a)=(aj',0)$ (where $j=f(j')$ since $J\subset f(A)$) is an $A\bowtie^{f}J$-homomorphism. \\
  Assume that $Ann_B(J)=0$. It is easy to see that $\phi$ is injective. It remains to verify that $\phi$ is surjective. Let $h: A\rightarrow A\bowtie^{f}J$ be an $A\bowtie^{f}J-$ homomorphism and it is determined by $h(1)=(x,y)$ where $x\in A,$ $y\in f(A)+J$ and $y-f(x)\in J$. Now, for all $j\in J,$ we have $(xa,yf(a))=h(a)=h((a,f(a)+j).1)=(a,f(a)+j)(x,y)=(ax,f(a)y+jy)$. So, $h$ is well defined if and only if $yj=0$ for all $j\in J$. Since $Ann_B(J)=0$, then $y=0$. Therefore, $h=\psi_{f(x)}$. Hence, $ J \cong Hom_{A\bowtie^{f}J}(A,A\bowtie^{f}J)$.
\end{proof}

\bigskip

\begin{proposition}\label{lemmm}
Let $A$ be a local ring, $B$ be a ring, $f : A \rightarrow B$ be a surjective ring homomorphism and $J$ be a non zero ideal of $B$.
Assume that $Ann_B(J)=0$. Then $A\bowtie^{f}J$ is quasi-Frobenius if and only if so is $A$ and $J=A$.
\end{proposition}

\begin{proof} Assume that $A\bowtie^{f}J$ is quasi-Frobenius. By Lemma
\ref{quas}, $A$ is quasi-Frobenius. Using Lemma \ref{prop0}, $ J
\cong_{A} Hom_{A\bowtie^{f}J}(A,A\bowtie^{f}J)$. So, $J$ is an
injective $A-$module since $A\bowtie^{f}J$ is self injective
(quasi-Frobenius). Then, by Lemma \ref{charquasi}, $J$ is
projective since $A$ is quasi-Frobenius (by Lemma
\ref{charquasi}). Since $A$ is local, $J$ is a regular principal
ideal. Let $z\in A$ be a regular element such that $J=zA$. The
following descendent
chain of ideals hold :\\
 ...$z^3A\subseteq z^2A\subseteq zA$. By Lemma \ref{charquasi}, $A$ is an artinian ring. Therefore, this chain is finite and so there is an
 integer $n$ such that $z^{n+1}=z^nA$. Then, there exists a non-zero element $y\in A$ such that $z^n=z^{n+1}y$ and $z^n(1-zy)=0$.
 Finally, $zy=1$, making $J=A$, as desired.
 \end{proof}

\bigskip

\bibliographystyle{amsplain}

\begin{thebibliography}{10}

\bibitem{a1} D.D. Anderson, \textit{Commutative rings}, in : Jim Brewer, Sarah Glaz, William Heinzer, Bruce Olberding (Eds.), Multiplicative Ideal Theory in Commutative Algebra: A tribute to the work of Robert Gilmer, Springer, New York, 2006, pp. 1-20.

\bibitem{benmahou} D. Bennis, N. Mahdou and K. Ouarghi, \textit{Rings over which all modules are strongly Gorenstein projective},
Rocky Mountain Journal of Mathematics, Vol. 40 (3) (2010),  749 -
759.

\bibitem{benmah} D. Bennis and N. Mahdou, \textit{Global Gorenstein Dimensions}, Proc. Amer. Math. Soc., Vol. 138 (2),
(February 2010), 461-465.

\bibitem{benmah2} D. Bennis and N. Mahdou, \textit{Global Gorenstein dimensions of polynomial rings and of direct products of rings},
Houston Journal of Mathematics 25 (4), (2009), 1019-1028.

 \bibitem{chhMahtam}M. Chhiti, N. Mahdou and M. Tamekkante, \textit{Self injective amalgamated duplication along an ideal}, J. Algebra Appl. 12, 1350033 (2013).

\bibitem{AFF1} M. D'Anna, C. A. Finacchiaro, and M. Fontana, \textit{Amalgamated algebras along an ideal},
Comm Algebra and Aplications, Walter De Gruyter (2009), 241--252.

 \bibitem{AFF2} M. D'Anna, C. A. Finacchiaro, and M. Fontana; \textit{Properties
 of chains of prime ideals in amalgamated algebras along
 an ideal}, J. Pure Applied Algebra {\bf 214}(2010), 1633-1641

\bibitem{A} M. D'Anna; \textit{A construction of Gorenstein rings}; J. Algebra {\bf 306}(2) (2006), 507-519.

\bibitem{AF1} M. D'Anna and M. Fontana; \textit{The amalgamated duplication
 of a ring along a multiplicative-canonical ideal}, Ark. Mat. {\bf 45}(2) (2007), 241-252.

\bibitem{AF2} M. D'Anna and M. Fontana; \textit{An amalgamated duplication
 of a ring along an ideal: the basic properties}, J. Algebra Appl.
{\bf 6}(3) (2007),  443-459.
\bibitem{do}J.L. Dorroh, \textit{Concerning adjunctions to algebras}, Bull. Amer. Math. Soc. 38 (1932), 85-88.


\bibitem{nich} W. K. Nicholson and M. F. Yousif, \textit{Quasi-Frobenius rings}, Cambridge University Press, 2003.
\bibitem{Nagata} M. Nagata, \textit{Local Rings}, Interscience, New York, 1962.

\end{thebibliography}

\end{document}